\documentclass[leqno,11pt]{amsart}
\newtheorem{thm}{Theorem}[section]
\newtheorem{cor}[thm]{Corollary}
\newtheorem{lemma}[thm]{Lemma}
\newtheorem{prop}[thm]{Proposition}

\theoremstyle{definition}

\theoremstyle{remark}

\newtheorem{notation}[thm]{Notation}

\renewcommand{\P}{\mathbb P}

\newcommand{\taub}{\mbox{\boldmath$\tau$}}

\newcommand{\ie}{i.~e.,\ }
\renewcommand{\Re}{\operatorname{Re}}

\begin{document}
\title[Extremal Polynomials]{On some classical problems concerning $L_{\infty}$-extremal polynomials with constraints}
\author{Franz Peherstorfer$^1$} 
\thanks{This work was supported
by the Austrian Science Fund FWF, project-number P20413-N18}
\thanks{$^1$The last modifications and corrections of this manuscript
were done by the author in the two months preceding this passing away
in November 2009. The manuscript remained unsubmitted and is not
published elsewhere (submitted by P. Yuditskii and I. Moale).}

\begin{abstract}
     First we consider the following problem which dates back to Chebyshev,
     Zolotarev and Achieser: among all trigonometric polynomials with given
     leading coefficients $a_0,...,a_l,$ $b_0,...,b_l \in \mathbb R$ find
     that one with least maximum norm on $[0,2 \pi].$ We show that the minimal
     polynomial is on $[0,2 \pi]$ asymptotically equal to a Blaschke product
     times a constant where the constant is the greatest singular value
     of the Hankel matrix associated with the $\tau_j = a_j + i b_j.$ As
     a special case corresponding statements for algebraic polynomials
     follow. Finally the minimal norm of certain linear functionals
     on the space of trigonometric polynomials is determined. As a consequence
     a conjecture by Clenshaw from the sixties on the behavior of the ratio
     of the truncated Fourier series and the minimum deviation is proved.
\end{abstract}

\maketitle

\section{Introduction}

In 1858 Chebyshev discovered that the polynomial
\begin{equation}
    2^{-n+1}T_n(x) = 2^{-n+1}\cos n \arccos x
    \label{eq-03}
\end{equation}
deviates least from zero with respect to the maximum norm on $[-1,1]$ among all
polynomials with leading coefficient one. Then he posed the following problem to
his circle: Let $l$ real numbers $A_0, A_1, \ldots, A_l$ be given. Among all
polynomials of degree less or equal $n$ with leading coefficients
$A_0, A_1, \ldots, A_l$, \ie of the form $\sum _{j=0}^l A_j x^{n-j} + q(x)$,
$q\in {\mathbb P}_{n-l-1},$ find that one which has least max-norm on $[-1,1]$,
that is, find the unique polynomial $\tilde{q}\in \P_{n-l-1}$ such that
\begin{equation}
    \min_{q\in \P_{n-l-1}}|| \sum _{j=0}^l A_jx^{n-j} + q(x)|| = || \sum _{j=0}^l
    A_jx^{n-j} + \tilde{q}(x)||
    \label{eq-01}
\end{equation}
where $||f||=\max_{x\in [-1,1]}|f(x)|$. This was the begin of a long story.
Indeed, ten years later Zolotarev, a student of Chebyshev, determined the
minimal polynomial in terms of elliptic functions when the first and second
coefficient is given. In 1930 Achieser gave a description of the minimal
polynomial in terms of automorphic functions when three leading coefficients
are given. In the words of Bernstein \cite[p. 156]{AKHG1} ``Akhieser
treated the more difficult problem and arrived at three algebraic equations
containing automorphic Schottky functions, whose solutions let to the
determination of the minimum deviation. Unfortunately these equations
are so complicated that it seems to be quite difficult to obtain simple
and sufficiently accurate inequalities''. Kolmogorov, Krein at al. also
mentioned in \cite[p. 233]{AKHG2} that the solution of the problem
was one of the significant contributions of Akhieser. Since in the explicit
representations there appear parameters given implicitly (as in Zolotarev's
representation the module of the elliptic functions) even for these two cases
there was (and is) still a demand for an asymptotic description in
elementary functions. Already in 1913 Bernstein himself
attacked the problem. For Zolotarev's case Bernstein succeeded in finding an
asymptotic solution of the minimum deviation, that is, of $E_n(A_0x^n +
A_1x^{n-1}) = \inf_{q\in \P_{n-2}}||A_0x^n + A_1x^{n-1} +q(x)||$ in terms of
elementary functions and also upper and lower bounds. In the sequel
he \cite{Ach} and later Achieser \cite{} obtained asymptotics of the
minimum deviation for some other special cases, for more recent results
on estimates of the minimum deviation see  Gutknecht and Trefethen \cite{GutTre},
where the error function is studied also, and Haussman and
Zeller \cite{Hau}. But neither Bernstein nor Achieser gave asymptotics
for the minimal polynomials in contrast to the $L_2$-norm where both
were main contributors in the development of an asymptotic theory.
Interesting enough the same holds for Szeg\H{o} the other great master
in asymptotics of orthogonal polynomials.

In the sixties N. N. Meiman \cite{Mei1, Mei2, Mei3} attacked the problem to
describe the minimal polynomial, called $Z_n$ in the following, when
$l,$ $l \in \mathbb N,$ coefficients are given.

By the Alternation Theorem $Z_n$ has at least $n-l$ alternation (abbreviated a-)
points on $[-1,1]$. Since every a-point from $(-1,1)$ is a critical point it
follows that the inverse image of $[-1,1]$ under $Z_n$ consists of at most
$l', $ $1 \leq l' \leq l+1,$ analytic arcs $\Gamma _j$, one of them,
say $\Gamma _0$ is the interval $[-1,1]$. Denoting the endpoints of the arcs by
$\alpha_{2,n}, \beta_{2,n}, \ldots, \alpha_{l',n}, \beta_{l',n}$ it
follows by the equioscillating property that the normed Zolotarev polynomial
$\tilde{Z}_n = Z_n / ||Z_n||$ satisfies
\begin{equation}
    \tilde{Z}_n^2 - 1 =  \left(\frac{\tilde{Z}_n'(x)}{\prod
    _{j=1}^{l'} (x-\gamma_{j,n})}\right)^2 \prod _{j=1}^{l'} (x-\alpha_{j,n})(x-\beta_{j,n})
    \label{eq-e0}
\end{equation}
that is,
\begin{equation}
    \frac{\tilde{Z}_n'^2}{\tilde{Z}_n^2 - 1} =
    \frac{\prod _{j=1}^{l'} (x-\gamma_{j,n})^2}
         {\prod _{j=1}^{l'} (x-\alpha_{j,n})(x-\beta_{j,n})}
    \label{eq-e1}
\end{equation}
where the $\gamma_{j,n}$'s are such that, $j = 0,...,l,$
\begin{equation}
         \int_{\alpha _{j,n}}^{\beta _{j,n}} \frac{\prod _{j=1}^{l'}
         (x-\gamma_{j,n})}{\sqrt{\prod _{j=1}^{l'} (x-\alpha_{j,n})
         (x-\beta_{j,n})}} dx =
         \frac{k_j\pi}{n}
        \label{eq-es4}
\end{equation}
where $k_j$ is the number of alternation points on the arc $\Gamma_j.$
\begin{equation}
   \tilde{Z}_n(t) =
   \pm \cosh (n \int_1^t
   \frac{\prod _{j=1}^{l'} (x-\gamma_{j,n})}
        {\sqrt{\prod_{j=1}^{l'} (x-\alpha_{j,n})(x-\beta_{j,n})}} dx).
   \label{eq-es6}
\end{equation}

Roughly speaking Meiman investigated in detail the precise number
$l'$ of arcs and gave a (more) detailed geometric description of
the arcs $\Gamma_j.$ For an explicit representation of the polynomial
$\tilde{Z}_n$ explicit expressions for the endpoints $\alpha_{j,n},$
$\beta_{j,n}$ and the zeros of the derivative $\gamma_{j,n}$ would
be needed. To find such explicit expressions is extremely unlikely
because one has to solve the system of hyperelliptic integrals
\eqref{eq-es4} and to find out how the endpoints of the arcs
are related to the given leading coefficients. In \cite{Mei1, Mei2, Mei3}
no way of solution is offered to this fundamental open question.

But let us observe that there is an interesting property of these points.
Since $\tilde{Z}_n$ has a finite number of zeros, precisely at most $l$
zeros, outside of $(-1,1)$ and since $\tilde{Z}_n$ is a minimal polynomial
on $\cup \Gamma_j,$ the length of each arc $\Gamma _j$, $j=1,\ldots,l$,
has to shrink to a point in the limit. Thus if we are interested in
asymptotics the problem reduces to find the connection between the $l$
given coefficients and the accumulation points of the $l'$-arcs or, in
other words, the $l'$-zeros of $\tilde{Z}_n$ lying outside $[-1,1];$
recall that every $\Gamma_j$ is a component of $\tilde{Z}^{-1}_n([-1,1]),$
for a description of inverse polynomial images see \cite{Peh4}.
To find this connection we proceed as follows. As usual we transform
the problem by the Joukowski-map to the complex plane such that the
interval $[-1,1]$ corresponds to the unit circle. Then we approximate
the polynomial of degree $l$ with the given $l$ coefficients by functions
from $H_\infty$ (so-called Caratheodory-Fejer approximation). This yields
a Blaschke product. Reflecting the Blaschke product at the unit circle it
turns out that its real part represents asymptotically the polynomial
$Z_n$ and, in particular, the zeros of the reflected Blaschke  product
are the limits of the arcs.

Roughly speaking we have shown that asymptotically there is a unique
correspondence between polynomials with $l$ fixed leading and of least
maximum norm on $[-1,1]$ and polynomials which vanish outside $[-1,1]$
at $l$ given points and are minimal on $[-1,1].$ Since we may expect
that outside $[-1,1]$ the polynomial grows exponentially fast we may
conclude that the minimal polynomial which vanishes at given $l$ points
represents asymptotically (up to a multiplication constant) every
polynomial satisfying in each of the $l$ given points any interpolation
condition (not depending on $n$). Indeed in this way we obtain asymptotic
representations of polynomials satisfying interpolation constraints
including constraints on the derivative. So far for special cases
asymptotics for the minimum deviation (but not for the minimal polynomial)
have been found by Bernstein \cite{}, see also \cite{}, $n$-th root
asymptotics has been derived by Fekete and Walsh \cite{FekWal}, see also
\cite{}.

Mostly it is more convenient to formulate the
problem in terms of Chebyshev polynomials, that is, to use the representation
\begin{equation}
    \sum _{j=0}^l A_jx^{n-j} = \sum _{j=0}^l a_jT_{n-j}(x) + q(x)
    \label{eq-02}
\end{equation}
$q\in \P_{n-l-1},$ where the first $l$ coefficients $a_j$ are given by the $A_j$'s.
In fact we will even study the more general problem of minimal trigonometric
polynomials with fixed leading coefficients; more precisely, denote by
${\mathcal T}_m = \{ \sum_{k = 0}^m a_k \cos k \varphi + b_k \sin k \varphi :
a_k,b_k \in \mathbb R\}$ the set of trigonometric polynomials of degree less
or equal $m,$ and let $a_0,...,a_l,b_0,...,b_l \in \mathbb R$ be given:
find the unique trigonometric polynomial
${\mathcal Z}_n(\varphi; a_0,...,a_l,b_0,...,b_l)$ for which
\begin{equation}\label{def_z}
\begin{aligned}
     & \min_{\substack{a_j,b_j \\ l+1 \leq j \leq n}}
       || \sum_{j = 0}^n a_j \cos (n - j)\varphi + b_j \sin (n - j)\varphi ||_{[0,2 \pi]} = \\
   = & ||{\mathcal Z}_n(\varphi; a_0,...,a_l,b_0,...,b_l) ||_{[0,2 \pi]}
\end{aligned}
\end{equation}
is attained; or in other words: given $\bar{\tau}_0 = a_0 - i b_0,...,$
$\bar{\tau}_l = a_l - i b_l \in \mathbb C$ find the unique polynomial
${\mathfrak Z}_n(z; \bar{\tau}_0,...,\bar{\tau}_l)$ of degree $n$ for which
\begin{equation}\label{pb}
     \min_{\substack{\tau_j \in \mathbb C \\ l+1 \leq j \leq n}} || {\rm Re} \{ \bar{\tau}_j e^{i(n - j)\varphi} \} || = \\
     || {\rm Re} \{ {\mathfrak Z}_n(e^{i \varphi}; \bar{\tau}_0,...,\bar{\tau}_l) \} ||
\end{equation}
is attained. Naturally
\begin{equation*}
     {\mathcal Z}_n(\varphi; \bar{\tau}_0,...,\bar{\tau}_l) :=
     {\mathcal Z}_n(\varphi; a_0,...,a_l,b_0,...,b_l) =
     {\rm Re} \{ {\mathfrak Z}_n(e^{i \varphi}; \bar{\tau}_0,...,\bar{\tau}_l) \}
\end{equation*}
Obviously when the $b_j$'s are zero then, using \eqref{eq-02}, we are back
in the algebraic case.

In the second part we consider linear functionals on the space of
truncated trigonometric polynomials and their applications, that is,
given $\mu_0,...,\mu_l \in \mathbb C$ how large can be the linear
functional $| \sum_{j = 0}^l \mu_{l - j} \tau_j |$ if
$|| {\rm Re} \{ \sum_{j = 0}^n \bar{\tau}_j e^{i(n - j)\varphi} \} ||$
$\leq 1.$ Note that the first $l + 1$ leading coefficients of the
trigonometric polynomial are given by $\bar{\tau}_0,\bar{\tau}_1,...,$
$\bar{\tau}_l$ and the remaining $n - l + 1$ coefficients are free
available. We will determine the least upper bound of
$|\sum_{j = 0}^l \mu_{l - j} \tau_j|$ for all $n \in \mathbb N.$

With the help of the solution of the problem just discussed we are able
to solve an old problem, see \cite{Cle,??} whether truncated Fourier
series can be used as a substitute for best approximations. A justification
of the method resulted in the conjecture of Clenshaw that
$|\sum_{j = 0}^l \tau_j \cos j \varphi|/E_n(f)$ is bounded by Landau's
constant as $n \to \infty.$ The articles by Clenshaw, Lam and Elliot and
Talbot \cite{Cle, LamEll, Tal} are devoted to show numerically that Clenshaw's
conjecture hold, at least for small $l,$ i.e. for $l = 1,2,3,4$ respectively
by giving algorithm for larger $l,$ for details concerning open and solved
questions of Clenshaw's conjecture, see \cite[p. 275]{Tal}. In contrast to
p. 275 the last but one paragraph in the introduction in \cite{Tal} is
misleading with this respect, partly incorrect, as the statement
``The first published proof of Clenshaw's conjecture was given by Lam and
Elliot \cite{LamEll} in 1972.'' There is no proof in \cite{LamEll} as mentioned
in \cite[p. 275]{Tal} also. The conjecture will follow as an easy consequence
of our derivations.

\section{Main Theorem}

%\begin{thm}
%Let $\tau _0,\tau_1, \ldots, \tau _l \in \R$ and let $|\mu|$ be the modulus
%of the greatest zero of the polynomial
%\[
%  \Delta (\lambda ):= \left| \begin{array}{lllcll}\tau _l-\lambda
%      & \tau _{l-1}
%      &\tau _{l-2} & \cdots &\tau _1 & \tau _0 \\
%  \tau _{l-1}
%      & \tau _{l-2} - \lambda & \tau _{l-3} & \cdots &\tau _2 &\tau _1\\
%        \multicolumn{6} {c} \dotfill \\
%  \tau _0
%      & {} & {} & \cdots  & {} & -\lambda \end{array} \right| \]
%and suppose that there are $\nu +1$ zeros of $\Delta (\lambda )$ with modulus
%$|\mu|$. Furthermore let $p_{l-\nu}$ be the polynomial of degree $l-\nu$, which
%has all zeros in $|z|<1$, such that
%\begin{equation}
%    \frac{1}{\mu}\frac{p_{l-\nu}(z)}{p_{l-\nu}^*(z)} = \tau_0 + \tau _1z
%    + \tau _2z^2 + \ldots + \tau _{l-\nu} z^{l-\nu}+ O(z^{l-\nu+1}).
%    \label{eq-th1}
%\end{equation}
%Then the trigonometric Zolotarev polynomial is asymptotically given on
%$[0,2\pi]$ by, $z=e^{i\varphi},$
%\begin{equation}
%    Z_n(\varphi,\bar{\tau} _0, \ldots, \bar{\tau} _l) = \frac{1}{\mu}\Re \{z^n
%    \frac{p_l^*(z)}{p_l(z)}\} + O(?)
%    \label{eq-th2}
%\end{equation}
%and
%\begin{equation}
%    ||Z_n(x;\tau _0,\ldots,\tau _l)|| \simeq \frac{1}{\mu}
%    \label{eq-th3}
%\end{equation}
%
%\end{thm}

%\begin{remark}
%Note that $p_{l-\nu}$ from \eqref{eq-th1} satisfies in fact the relation
%\begin{equation}
%    \frac{1}{\mu}\frac{p_{l-\nu}(z)}{p_{l-\nu}^*(z)} =
%    \sum_{j=0}^l \bar{\tau} _jz^j + O(z^{l+1})
%    \label{eq-th4}
%\end{equation}
%\end{remark}

\begin{thm}\label{three}
Let $p_l(z) = z^l + ...$ be a polynomial of degree $l$ which has all its zeros in
$|z|<1$ and the expansion at $z=0$
\begin{equation}
     \frac{p_l(z)}{p_l^*(z)} = \tau_0 + \tau _1z + \tau_2 z^2 +
     \ldots + \tau_m z^m + O(z^{m+1})
     \label{eq-th5}
\end{equation}
where $m \geq l.$ Then on $[0,2\pi]$ the trigonometric Zolotarev polynomials
of degree $n$ with leading coefficients $\bar{\tau}_0, ... ,\bar{\tau}_k,$
$k = l,...,m,$ $k \leq n,$ are given asymptotically by
\begin{equation}
    \mathcal{Z}_n(\varphi; \bar{\tau} _0,\bar{\tau} _1, \ldots ,\bar{\tau} _k)
    = \Re \{ z^n \frac{p_l^*(z)}{p_l(z)}\} + O(\tilde{r}^n),
    \label{eq-th6}
\end{equation}
where $\tilde{r}>r:=\max\{|z_j|: z_j \;\text{zero of}\;\; p_l\}$ and the constant
in the $O( \ )$ term does not depend on $n.$ Moreover
\begin{equation*}
     || {\mathcal Z}_n(\varphi; \bar{\tau}_0,\bar{\tau}_1,...,\bar{\tau}_k) || \sim 1
\end{equation*}
with geometric convergence.
\end{thm}

Next let us show that condition \eqref{eq-th5} is satisfied always, because
of the following Theorem which goes back to Caratheodory and Fejer \cite{}
and Schur \cite{Sch}.

\begin{thm}\label{Sch}{\rm (}\cite{Sch}{\rm)}
Let $\tau_0,...,\tau_m \in \mathbb C$ be given. Then there exists a
$l \in {\mathbb N}_0,$ $0 \leq l \leq m,$ a polynomial $p_l(z)$ of
degree $l$ and a $\gamma \in \mathbb C$ such that $p_l(z) = z^l + ...$
has all zeros in $|z| < 1$ and
\begin{equation}\label{eq-schur}
   \gamma \frac{p_l(z)}{p_l^*(z)} = \tau_0 + ... + \tau_m z^m + O(z^{m + 1}),
\end{equation}
where $\gamma$ is the zero of largest modulus of
\[
  D_{l + 1} (\lambda ):=
  \left|
  \begin{array}{cccccccl}
  \lambda      & 0                 & \cdots & 0            & \tau _0 & \tau_1  & \cdots & \tau_l     \\
  0            & \lambda           & \cdots & 0            & 0       & \tau_0  & \cdots & \tau_{l-1} \\
  \multicolumn{6} {c} \dotfill                                                                       \\
  0            & 0                 & \cdots & \lambda      & 0       & 0       & \cdots & \tau_0     \\
  \bar{\tau}_0 & 0                 & \cdots & 0            & \lambda & 0       & \cdots & 0          \\
  \bar{\tau}_1 & \bar{\tau}_0      & \cdots & 0            & 0       & \lambda & \cdots & 0          \\
  \multicolumn{6} {c} \dotfill                                                                       \\
  \bar{\tau}_l & \bar{\tau}_{l -1} & \cdots & \bar{\tau}_0 & 0       & 0       & \cdots & \lambda    \\
  \end{array}
  \right|
\]
and $D_j(\gamma) \neq 0$ for $j = 1,...,l.$ If $\tau_0,...,\tau_l \in \mathbb R$
then
\begin{equation*}
    D_{l + 1}(\lambda) = \Delta(\lambda) \Delta(- \lambda)
\end{equation*}
where
\[
  \Delta (\lambda ):=
  \left|
  \begin{array}{lllcll}
     \tau _l-\lambda & \tau _{l-1}           & \tau _{l-2} & \cdots & \tau _1 & \tau _0 \\
     \tau _{l-1}     & \tau _{l-2} - \lambda & \tau _{l-3} & \cdots & \tau _2 & \tau _1 \\
     \multicolumn{6} {c} \dotfill                                                       \\
     \tau _0         & {}                    & {}          & \cdots & {}      & -\lambda
  \end{array}
  \right|
\]
is the characteristic polynomial of the associated Hankel matrix.
\end{thm}

Combining Theorem \ref{three} and Theorem \ref{Sch} we obtain
that for given $\tau_0,...,$ $\tau_m \in \mathbb C$ there exists
a $l \in {\mathbb N}_0,$ $0 \leq l \leq m$ and $\gamma \in \mathbb C$
such that for $\varphi \in [0,2 \pi]$ and for $k = l,...,m$
\begin{equation}\label{eq-x1}
   {\mathcal Z}_n(\varphi; \bar{\tau}_0,...,\bar{\tau}_k) =
   |\gamma| {\rm Re} \{ e^{i \arg \gamma} z^n \frac{p_l^*(z)}{p_l(z)} \} +
   O(\tilde{r}^n),
\end{equation}
where the constant in the $O( \ )$ term does not depend on $n,$
and
\begin{equation*}
    || {\mathcal Z}_n(\varphi; \bar{\tau}_0,...,\bar{\tau}_k) || \sim |\gamma|
\end{equation*}
where the convergence is geometric. $l,$ $p_l$ and $\gamma$ are
determined by the above Theorem.

Thus the problem of determining the trigonometric and algebraic
Zolotarev polynomials with an arbitrary given number of leading
coefficients with respect to the maximum norm is completely
solved asymptotically by the above Theorems.

We mention that by Fej\'{e}r \cite{} the condition
$\tau_l \geq \tau_{l - 1} \geq ... \geq \tau_1 \geq \tau_0 > 0$
implies that $l = m$ in Theorem \cite{CarFej}
and that the coefficients of $p_l(z) = \sum_{\nu = 0}^l a_\nu z^{l - \nu}$
satisfy $a_0 \geq a_1 \geq ... \geq a_l \geq 0.$

For given ${\boldsymbol \tau}$ only a few explicit values of
$|\gamma({\boldsymbol \tau})| =: \tilde{\gamma}({\boldsymbol \tau})$
are known, for instance
\begin{equation*}
    \tilde{\gamma}((1,...,1)) = 1/2 \sin \left( \frac{\pi}{2(2 l + 3)} \right)
\end{equation*}
which gives by \eqref{eq-x1} that
\begin{equation*}
\begin{split}
    & {\mathcal Z}_n(\varphi; (1,...,1)) = \frac{1}{2 \sin \left( \frac{\pi}{2(2 l + 3)} \right)} . \\
    & \qquad \qquad \qquad \qquad {\rm Re \ }
      \left\{
            \frac{\sin \frac{(l + 1)\pi}{2 l + 3} + \sin \frac{l\pi}{2 l + 3}z + \cdots + \sin \frac{\pi}{2 l + 3}z^l}
                 {\sin \frac{\pi}{2 l + 3} + \sin \frac{2 \pi}{2 l + 3}z + \cdots + \sin \frac{(l + 1)\pi}{2 l + 3}z^l }
      \right\}
\end{split}
\end{equation*}

%%%%%%%%%%%%%%%%%%%%%%%%%%%%%%%%%%%%%%%%%%%%%%%%%%%%%%%%%%%%%%%%%%%%%%%%%%%%%%%%%%%%%%%%%%%%%%%%%%%
\section{Proofs}

We will prove the following more general version of
Theorem \ref{three}:

\begin{thm}\label{three-modified}
Let $(n_\nu)$ be a subsequence of $\mathbb N$ and let
$p_{l,n_\nu}(z) = z^l + ...$ be such that
\begin{equation}\label{m1}
       \frac{p_{l,n_\nu}(z)}{p_{l,n_\nu}^*(z)} =
       \tau_{0,n_\nu} +
       \tau _{1,n_\nu} z +
       \ldots +
       \tau _{m,n_\nu} z^m +
       O(z^{m+1}),
\end{equation}
where $m$ is independent of $n_\nu$ and $m \leq l,$ and that
\begin{equation}\label{x}
       p_{l,n_\nu}(z) \underset {n_\nu \to \infty} \longrightarrow p_{l}(z),
       \ {\rm where \ } p_l(z) \ {\rm has \ all \ zeros \ in \ } |z| \leq r < 1.
\end{equation}
Then for $\varphi \in [0,2\pi]$ and $k = l,...,m,$
\begin{equation}\label{m2}
    \begin{aligned}
       {\mathcal Z}_{n_\nu}(\varphi; \bar{\tau}_{0,n_\nu}, ... , \bar{\tau}_{k,n_\nu})
       & = \Re \{ z^{n_\nu} \frac{p_{l, n_\nu}^*(z)}{p_{l, n_\nu}(z)} \} +
         O(\tilde{r}^n) \\
       & = \Re \{ z^{n_\nu} \frac{p_l^*(z)}{p_l(z)} \} +
         O(\tilde{\tilde{r}}^n) \\
    \end{aligned}
\end{equation}
where $\tilde{r} > r$ and the constant in the $O( \ )$ term does not depend on $n.$
Furthermore,
\begin{equation}\label{m2}
      ||
         {\mathcal Z}_{n_\nu}(\varphi;
         \bar{\tau}_{0,n_\nu}, ... ,\bar{\tau}_{k,n_\nu})
      || \sim 1.
\end{equation}
\end{thm}

Notation: Let $l,n \in \mathbb N$ and suppose that $p_{l,n}(z) = z^l + ...$
has no zero on $|z|=1$. In the following let, $z=e^{i \varphi}$,
\begin{equation}\label{proofs1}
   R_n(\varphi) = {\rm Re \ } \{ z^{n}
   \frac {p_{l,n}^{*}(z)} {p_{l,n}(z)} \} =
   \frac { {\rm Re \ } \{ z^{n-l}(p_{l,n}^{*}(z))^2 \} } { |p_{l,n}(z)|^2 }
\end{equation}
and
\begin{equation}\label{proofs2}
   S_n(\varphi) = {\rm Im \ } \{ z^{n} \frac {p_{l,n}^{*}(z)} {p_{l,n}(z)} \} =
   \frac { {\rm Im \ } \{ z^{n-l}(p_{l,n}^{*}(z))^2 \} } { |p_{l,n}(z)|^2 }
\end{equation}
Obviously $R_n(\varphi)$ and $S_n(\varphi)$ are rational trigonometric
functions with a trigonometric polynomial of degree $n+l$ in the numerator
and a positive trigonometric polynomial of degree $l$ in the denominator.
Note that
\begin{equation}\label{insert}
   R_n^2(\varphi) + S_n^2(\varphi) = 1
\end{equation}

\begin{lemma}\label{proofslemma2.2}
Let $l,n \in \mathbb N$ with $n>l$ and suppose that $p_{l,n}(z)$ has all
zeros in $|z|\leq \tilde{r}<1$ for all $n \geq n_0.$ Then the following
statements hold for every $n \geq n_1:$ a) Both $R_n$ and $S_n$ have
exactly $2(n-l)$ simple zeros in $[0,2\pi)$ and their zeros strictly
interlace. Furthermore, $R_n(S_n)$ has $2(n-l)$ a-points at the zeros
of $S_n(R_n).$ In particular, $0$ is a best approximation to $R_n$ and
$S_n$ with respect to $\frak{T}_{n-l-1}.$

b) The numerators in \eqref{proofs1} and \eqref{proofs2} can be
represented as follows, $z=e^{i \varphi}$,
\begin{equation}\label{proofs3}
    {\rm Re \ } \{ z^{n-l}(p_{l,n}^{*}(z))^2 \} =
    t_{n-l,n}(\varphi) |\hat{r}_{2l,n}(e^{i \varphi})|^2
\end{equation}
and
\begin{equation}\label{proofs4}
    {\rm Im \ } \{ z^{n-l}(p_{l,n}^{*}(z))^2 \} =
    u_{n-l,n}(\varphi) |\hat{s}_{2l,n}(e^{i \varphi})|^2
\end{equation}
where $t_{n-l}(\varphi), u_{n-l}(\varphi)$ are trigonometric
polynomials of degree $n-l$ which have all their $2(n-l)$ zeros
in $[0, 2\pi)$ and their zeros strictly interlace. Furthermore
$\hat{r}_{2l}(z)$ and $\hat{s}_{2l}(z)$ are monic polynomials
of degree $2l$ which have all their zeros in $|z| < 1.$
\end{lemma}

\begin{proof}
For simplicity of writing we omit the index $n.$
Concerning part a). Since $p_l^{*}(z)$ has all zeros in
$|z| \geq 1/r >1$ it follows that for sufficiently large $n$ the
$\arg(z^{n}\frac{p_l^{*}(z)}{p_l(z)})$ is strictly monotone
increasing with respect to $\varphi$, where the change of the
argument is $2(n-l)\pi$ when $\varphi$ varies from $0$ to $2\pi.$
Since ${\rm Re } \{ z^{n} \frac{p_l^{*}(z)}{p_l(z)} \}$ and
${\rm Im } \{ z^n \frac{p_l^{*}(z)}{p_l(z)} \}$ are zero, respectively,
if and only if $\arg z^{n} p_l^{*}(z)/p_l(z) = (2k+1)\pi/2$
respectively $k \pi$, where $k \in \mathbb Z,$ the statement
about the number of zeros and on the interlacing property follows.

b) The statements follow immediately by part a) and the F\'ejer-Riesz representation
for nonnegative trigonometric polynomials.
\end{proof}

\begin{lemma} \label{lemma2.3}
Under the assumptions of Theorem \ref{three-modified}
the polynomials $\hat{r}_{2l,n}(z)$ and $\hat{s}_{2l,n}(z)$
associated with $p_{l,n}^{*}(z)$ by \eqref{proofs3} and
\eqref{proofs4}, respectively, satisfy on any compact subset
of $\mathbb C$
\begin{equation}\label{proofsl2}
\begin{aligned}
    \hat{r}_{2l,n}(z) = (p_l(z))^2 + O(r^n) \\
    \hat{s}_{2l,n}(z) = (p_l(z))^2 + O(r^n)
\end{aligned}
\end{equation}
where $p_l$ is given by \eqref{x} and where $0 < r < 1.$
\end{lemma}

\begin{proof}
Let us consider the zeros of the polynomial
\begin{equation}\label{proofsx}
   P_{2n+2l}(z) = z^{2n}(p_{l,n}^{*})^2(z) + (p_{l,n})^2(z).
\end{equation}
Note that, $z=e^{i \varphi},$
\begin{equation}
   P_{2n+2l}(e^{i\varphi}) =
   z^{n + l} {\rm Re \ } \{ z^{n-l}(p_{l,n}^{*}(z))^2 \} =
   z^{n - l} t_{n - l,n}(\varphi) r_{2l, n}(z) r^*_{2 l, n} (z)
\end{equation}
Thus by Lemma \ref{proofslemma2.2} $P_{2n+2l}(z)$ has $2n-2l$
zeros on $|z|=1$ and by the self-reciprocal property $2l$ zeros in
$|z| < 1$ and in $|z| > 1$ which are the zeros of
$\hat{r}_{2l,n}(z)$ and $\hat{r}_{2l,n}^{*}(z),$ respectively. By
assumption \eqref{x}, \eqref{proofsx} and Rouch\'e's theorem it follows
that for $n \geq n_0$ $\hat{r}_{2l,n}(z)$ has (exactly) two zeros in each
neighborhood of a zero of $p_l.$ More precisely, if $v_{j,n}$ is a
zero of $P_{2n+2l}$ from the neighborhood of a zero $z_j$ of $p_l$
we have, by \eqref{proofsx} again, that $|v_{j,n} - z_j| = O(r^n)$
for some $r, 0 < r < 1,$ where $r$ is independent of $j$ and $n.$
Hence
$$ \hat{r}_{2l,n}(z) = p^2_{l}(z) + O(r^n). $$
Analogously the statement for $\hat{s}_{2l,n}$ is proved.
\end{proof}

\begin{lemma}\label{lemma2.4}
Under the assumption of Theorem \ref{three-modified} $R_n$ defined in \eqref{proofs1} is of the form
\begin{equation}\label{proofstilda1}
   R_n(\varphi) = V_n(\varphi) + \psi_n(\varphi)
\end{equation}
where $V_n \in \frak{T}_n$ with $V_n(\varphi) = {\rm Re \ }
\{ \sum_{j=0}^m \bar{\tau}_{j,n} z^{n-j} \} + ...$ and $\psi_n(\varphi)$
is a rational trigonometric function  with
\begin{equation}\label{proofstilda2}
   ||\psi_n|| = O({\tilde{r}}^n)
\end{equation}
where $\max |z_j| < \tilde{r} < 1,$ the $z_j$'s are the zeros of $p_l,$
and the constant in the $O( \ )$ term does not depend on $n.$
\end{lemma}

\begin{proof}
By Euclid $R_n$ from \eqref{proofs1} can be written in the form
\begin{equation}\label{xx}
   R_n(\varphi) = \frac{V_n(\varphi) |p_{l,n}(e^{i \varphi})|^2 +
   t(\varphi)}{|p_{l,n}(e^{i \varphi})|^2},
\end{equation}
where $V_n \in \mathfrak{T}_n$ and $t \in \mathfrak{T}_{l-1}.$
Putting, $z = e^{i \varphi},$
\begin{equation}\label{proofstilda3}
   e^{in\varphi}V_n(\varphi) = z^n P_n^*(z) + P_n(z)
\end{equation}
we obtain by \eqref{proofs1} and partial fraction
expansion, assuming that $p_{l,n}$ has simple zeros, that
\begin{equation}\label{proofstilda4}
\begin{aligned}
    \frac{z^{2n}(p_{l,n}^*)^2(z) + p_{l,n}^2(z)}{p_{l,n}(z) p_{l,n}^*(z)}
    & = z^n P_n^*(z) + P_n(z) + \\
    & + z^{n} \left( \sum_{j=1}^l \frac{\lambda_{j,n}}{z-z_{j,n}} +
      \sum_{j=1}^l \frac{\beta_{j,n}}{z-\frac{1}{ \overline{z}_{j,n} }}, \right)
\end{aligned}
\end{equation}
where
$$ \lambda_{j,n} = z_{j,n}^{n} \frac{p_{l,n}^{*}(z_{j,n})}{p_{l,n}^{'}(z_{j,n})}
   = O(\tilde{r}^n) $$
and analogously
$$ \beta_{j,n} =
      \bar{z}_{j,n}^{n}
        \frac
           { p_{l,n} ( \frac{1}{\bar{z}_{j,n}} ) }
           { (p_{l,n}^{*})^{'} ( \frac{1}{ \bar{z}_{j,n}} ) }
   = O(\tilde{r}^n)
$$
where for the last equalities we took \eqref{x} into consideration.
By \eqref{xx} and \eqref{proofstilda4} it follows that
$$ z^{2n} \frac{p_{l,n}^{*}}{p_{l,n}} + \frac{p_{l,n}}{p_{l,n}^{*}} =
   z^{n} P_{n}^{*} + P_{n} + O(z^{n+l}), $$
hence by \eqref{m1}
   $$ P_n(z) = \sum_{j=0}^m \tau_{j,n} z^j + ... $$
which gives by \eqref{proofstilda3} the assertion on the leading
coefficients of $V_n.$
\end{proof}

\begin{notation} Let $G$ be a linear space of $C[a,b]$ and $f\in C[a,b]$ with
$0$ as a best approximation. Denote by $\gamma (f,G)$ the strong unicity
constant of $f$ with respect to $G$ \ie
\begin{equation}
     \begin{aligned}
                    \gamma (f,G) & = \inf \frac{||f-g|| - ||f||}{||g||}\\
                    & = \inf_{||g||=1} \max_{y\in E(f)} {\rm sgn}(f(y)) g(y)
     \end{aligned}
     \label{eq-pl6}
\end{equation}
where $E(f)$ denotes the set of extremal points. If $G$ is a Haar space then by
\cite{Cli}
\begin{equation}
    \gamma (f,G) = 1/\max_{1\leq k\leq n}||g_k||
    \label{eq-pl7}
\end{equation}
where $g_k \in G$ is the polynomial given by $g_k(y_j) = (-1)^j$,
$j=1,\ldots,n+1; j\neq k$ and the $y_j$ denote the $a$-points of $f$.
\end{notation}

\begin{lemma}\label{lemma2.5}
$$1/\gamma (R_n, \frak{T}_{n-l-1}) = O(n) $$
\end{lemma}

\begin{proof}
Let
\begin{equation}\label{proofsx-1}
  {\mathfrak G}_{2(n-l)-1} := \left\{ q(\varphi) \frac{ |\hat{r}_{2l,n}(e^{i\varphi})|^2 }
  { |p_{l,n}(e^{i\varphi})|^2 } : q \in {\mathfrak T}_{n-l-1} \right\}
\end{equation}
where $\hat{r}_{2l,n}$ is the polynomial associated with $p_{l,n}$ by
\eqref{proofs3}. By Lemma \ref{proofslemma2.2} $t_{n-l,n}$ can be written
in the form, $z=e^{i \varphi},$
\begin{equation}\label{proofsx0}
  \begin{aligned}
     t_{n-l,n}(\varphi) = {\rm Re \ } \{ c_n z^{n-l} + ... \}
     & = \frac{c_n}{2} z^{-(n-l)} \prod_{\nu=1}^{2(n-l)} (z-e^{i \psi_{\nu,n}}) \\
     & = \frac{|c_n|}{2} (2i)^{2(n-l)} \prod_{\nu=1}^{2(n-l)}
       \sin \frac{\varphi-\varphi_{\nu,n}}{2}
  \end{aligned}
\end{equation}
where we used the fact that
\begin{equation}\label{19_2}
   e^{i \arg c_n} = e^{-i \sum\limits_{\nu=1}^{2(n-l)} \frac{\psi_{\nu,n}}{2}}
\end{equation}
Note that by \eqref{proofsx0} and \eqref{19_2}
\begin{equation}\label{19_3}
   t_{n-l,n}(\varphi) = |c_n| ( \cos ((n-l)\varphi -
   \sum_{\nu=1}^{2(n-l)} \frac{\psi_{\nu,n}}{2} ) + q(\varphi) )
\end{equation}
where $q \in \mathfrak{T}_{n-l-1}.$
Analogously we obtain
\begin{equation}\label{proofu}
  \begin{aligned}
   u_{n-l,n}(\varphi) = {\rm Im \ } \{ d_n z^{n-l} + ... \}
   & = \frac{d_n}{2} z^{-(n-l)} \prod\limits_{\nu=0}^{2(n-l)}
       (z-e^{i \varphi_{\nu,n}}) \\
   & = \frac{|d_n|}{2} (2i)^{2(n-l)} \prod\limits_{\nu=1}^{2(n-l)}
       \sin \frac{\varphi-\varphi_{\nu,n}}{2}
  \end{aligned}
\end{equation}
where $$e^{i \arg d_n} = e^{i \left( \frac{\pi}{2} -
\sum\limits_{\nu=1}^{2(n-l)} \frac{\varphi_{\nu,n}}{2} \right) }, $$
moreover
$$ u_{n-l,n}(\varphi) = |d_n|  \sin ((n-l)\varphi + \frac{\pi}{2} -
\sum_{\nu=1}^{2(n-l)} \frac{\varphi_{\nu,n}}{2} )  $$
Taking into consideration \eqref{proofsl2} we have
\begin{equation}\label{proofsx1}
    \lim\limits_n \ d_n = \lim\limits_n \ c_n \neq 0
\end{equation}

Next let us put for $k = 1, ... , 2n-2l, k \neq 1,$
\begin{equation}
\begin{split}
       q_{k,n}(\varphi) =
       & t_{n-l,n}(\varphi) + \\
       & \frac{|c_n|}{|d_n|}
         \frac{u_{n-l,n}(\varphi) \left( \sin \left( \varphi +
         \sum\limits_{\nu=1, \nu \neq 1,k}^{2(n-l)}
         \frac{\varphi_{\nu,n}}{2} -
         \sum\limits_{\nu=1}^{2(n-l)} \frac{\psi_{\nu,n}}{2} - \frac{\pi}{2} \right)
         - e_{k,n} \right) }
         { (-2) \sin \left( \frac{\varphi - \varphi_{1,n}}{2} \right)
         \sin \left( \frac{\varphi - \varphi_{k,n}}{2} \right) }
\end{split}
\end{equation}
where
\begin{equation}\label{proofsx2}
   e_{k,n} = \sin \left( \varphi_{1,n} +
   \sum_{\nu=1,\nu \neq 1,k}^{2(n-l)} \frac{\varphi_{\nu,n}}{2} -
   \sum_{\nu=1}^{2(n-l)} \frac{\psi_{\nu,n}}{2}
   - \frac{\pi}{2} \right)
\end{equation}
i.e. the second factor of the numerator is of the form
$const$ $\sin \left( \frac{\varphi - \varphi_{1,n}}{2} \right)$
$\sin \left( \frac{\varphi - \eta}{2} \right).$ Moreover the
second expression at the right hand side in \eqref{proofsx}
is from $\frak{T}_{n-l}.$ Since
$$ \frac{u_{n-l,n}(\varphi)}{(2i)^2 \sin \left( \frac{\varphi -
   \varphi_{1,n}}{2} \right) \sin \left( \frac{\varphi -
   \varphi_{k,n}}{2} \right)} =
   |d_n| {\rm Im \ } \{ e^{i( (n-l-1) \varphi + \frac{\pi}{2}-
   \sum\limits_{\nu=1, \nu \neq 1,k}^{2(n-l)} \frac{\varphi_{\nu,n}}{2}
   )} \}
$$
it follows by straightforward calculation that the second
expression at the right hand side of \eqref{proofsx} is of the form
$-|c_n| \cos \left( (n-l)\varphi - \sum\limits_{\nu=1}^{2(n-l)}
\frac{\psi_{\nu,n}}{2} \right) + q, q \in \frak{T}_{n-l-1}$;
hence it follows by \eqref{19_3} that $q_{k,n} \in \frak{T}_{n-l-1}$
and therefore
\begin{equation}\label{}
   g_{k,n}(\varphi) = q_{k,n}(\varphi) \frac{|\hat{r}_{2l,n}(e^{i\varphi})|^2}
   {|p_{l,n}(e^{i\varphi})|^2}
\end{equation}
has the properties that $g_{k,n} \in {\mathfrak G}_{2(n-l)-1}$ and, by \eqref{proofsx}
and Lemma \ref{proofslemma2.2} a)
\begin{equation}\label{}
   \pm g_{k,n}(\varphi_{\nu,n}) = (-1)^{\nu}  \ \ \ \ \nu = 1, ..., 2(n-l), \nu \neq k.
\end{equation}
Furthermore,
\begin{equation}\label{}
   h_{k,n}(\varphi) = \sin \left( \frac{\varphi - \varphi_{k,n}}{2} \right)
   g_{k,n}(\varphi)
\end{equation}
is uniformly bounded on $[0,2\pi]$ with respect to $n.$ Indeed,
$t_{n-l,n}|\hat{r}_{2l,n}|^2 / |p_{l,n}|^2$ $ = R_n(\varphi)$ and
$u_{n-l,n}|\hat{s}_{2l,n}|^2 / |p_{l,n}|^2 = S_n(\varphi)$ are
bounded by one. Further $|\hat{r}_{2l,n} / p_{l,n}|$ as well as,
recall \eqref{proofsl2}, $| \hat{r}_{2l,n} / \hat{s}_{2l,n} |$
are uniformly bounded on $|z|=1$, since $p_l(e^{i\varphi})$ is,
by assumption, bounded away from zero on $|z|=1,$ the boundedness
of $h_{k,n}$ follows by the choice \eqref{proofsx2} of $e_{k,n}$
in \eqref{proofsx} and \eqref{proofsx1}. Now $h_{k,n}(\varphi)$
is a trigonometric polynomial of half argument for which Bernstein's
inequality for the derivative still holds, therefore
\begin{equation}\label{}
\begin{aligned}
   ||g_{k,n}||
   & = \left|\left| \frac {h_{k,n}(\varphi) -
       h_{k,n}(\varphi_k)} {\varphi - \varphi_k} .
       \frac {\varphi - \varphi_k}
       { \sin \left( \frac{\varphi - \varphi_k}{2} \right) } \right|\right| \\
   & \leq ||h_{k,n}|| \left|\left| \frac {\varphi - \varphi_k}
     { \sin \left( \frac{\varphi - \varphi_k}{2} \right) } \right|\right| \leq
const. n
\end{aligned}
\end{equation}
Thus
$$1/\gamma (R_n; {\mathfrak G}_{2(n-l)-1}) \leq const.n  $$
Using the simple fact that $0$ is a best approximation to $R_n$ from
${\mathfrak G}_{2(n-l)-1}$ as well as from $\frak{T}_{n-l-1}$ we obtain by \eqref{} and
\eqref{proofsx-1} that
$$ \gamma(R_n;{\mathfrak G}_{2(n-l)-1}) \leq  ||\frac{\hat{r}_{2l,n}}{p_{l,n}}||^2
\gamma(R_n;\frak{T}_{n-l-1}) $$
hence
$$ 1/\gamma (R_n; \frak{T}_{n-l-1}) \leq \widetilde{const}. n   $$
\end{proof}

%%%%%%%%%%%%%%%%%%%%%%%%%%%%%%%%%%%%%%%%%%%%%%%%%%%%%%%%%%%%%%%%%%%%%%%%%%%%%%%%%%%%%%%%%%%%%%%%%%%

%%%%%%%%%%%%%%%%%%%%%%%%%%%%%%%%%%%%%%%%%%%%%%%%%%%%%%%%%%%%%%%%%%%%%%%%%%%%%%%%%%%%%%%%%%%%%%%%%%%%%%%%%%%%%%%%%%%%%%%%%%%%%%
\section{Proof of Theorem \ref{three} and Theorem \ref{three-modified}}

\begin{prop}\label{prop4.1}
Let $n,j(n) \in \mathbb N.$ Let $\tilde{t}_{j(n)}(\varphi)$ be a best
approximation to $f_n \in C_{2\pi}$ on $[0,2\pi)$ with respect to the
linear subspace $G_{j(n)}$ and let $0$ be a best approximation from
$G_{j(n)}$ to $h_n \in C_{2\pi}$ on $[0,2\pi).$ Suppose that $||h_n||=1$
and that
\begin{equation}\label{sa11}
   f_n(\varphi) - t_{j(n)}(\varphi) = h_n(\varphi) + \epsilon_n(\varphi),
\end{equation}
where $t_{j(n)} \in G_{m(n)}.$ If $\gamma(h_n;G_{j(n)}) ||\epsilon_n||
\underset { n \to \infty } \longrightarrow 0$ then for $\varphi \in [0, 2 \pi]$
\begin{equation}\label{sa12}
   f_n(\varphi) - \tilde{t}_{j(n)}(\varphi) = h_n(\varphi) +
   O (||\epsilon_n||\gamma(h_n;G_{j(n)}))
\end{equation}
and in particular
$$ t_{j(n)}(\varphi) - \tilde{t}_{j(n)}(\varphi) =
   O (||\epsilon_n||(1+\gamma(h_n;G_{j(n)})))$$
\end{prop}

\begin{proof}
\begin{equation}\label{sa13}
   |E_{j(n)}(f_n) - 1| \leq ||\epsilon_n||
\end{equation}
since $$ E_{j(n)}(f_n) \leq ||f_n - t_{j(n)}|| \leq ||h_n|| +
||\epsilon_n|| = 1 + ||\epsilon_n|| $$
and
$$ 1 = E_{j(n)}(h_n) = ||h_n|| \leq ||f_n - t_{j(n)} -
   \epsilon_n - ( \tilde{t}_{j(n)} - t_{j(n)} )|| \leq E_{j(n)}(f_n) + ||\epsilon_n||$$
Next let us show that
$$ ||\tilde{t}_{j(n)} - t_{j(n)}|| \leq ( 1 + 2\gamma(h_n;G_{j(n)}) ) ||\epsilon_n||$$
Indeed, by the definition \eqref{eq-pl6} of the strong uniqueness constant
$$\begin{aligned}
      ||\tilde{t}_{j(n)} - t_{j(n)}||
      & \leq \gamma(h_n;G_{j(n)})
        \left( ||h_n - (\tilde{t}_{j(n)} - t_{j(n)}) || - ||h_n|| \right) \\
      & \leq \gamma(h_n;G_{j(n)})
        (E_{j(n)}(f_n) + ||\epsilon_n|| - ||h_n|| )  \leq 2
        \gamma(h_n;G_{j(n)}) ||\epsilon_n||
\end{aligned}$$
where we have used \eqref{sa11} and the fact that $||h_n|| = 1$ and
\eqref{sa13} in the second and third inequality, respectively.
\end{proof}

\begin{proof}[of Theorem \ref{three}] Put $k(n) = n - k - 1,
G_{k(n)} = {\mathfrak T}_{n - k - 1}$ and $k \in \{ l, ..., m \}$ fixed,
and
$$ h_n(\varphi) = R_n(\varphi) \ {\rm and \ } f_n(\varphi) =
   {\rm Re \ } \{ \sum\limits_{j=0}^k \bar{\tau}_{j,n} z^{n-j} \}.$$
Recall that $0$ is a best approximation with respect to
${\mathcal T}_{n - l - 1}$ hence with respect to
${\mathcal T}_{n - k - 1}$ for $k \in \{ l, ..., n - 1 \}.$ By
Lemma \ref{lemma2.4},
$$ V_n = f_n - t_{k(n)} \ {\rm and \ } R_n = f_n - t_{k(n)} + \psi$$
where $t_{k(n)} \in {\mathfrak T}_{n - j - 1}$ and $||\psi_n|| = O(r^n).$
Since $f_n - \tilde{t}_{k(n)} = {\mathcal Z}_n(x; \bar{\tau}_{0,n} , ... ,
\bar{\tau}_{j,n} )$ Theorem \ref{three-modified} follows by Proposition
\ref{prop4.1} in conjunction with Lemma \ref{proofslemma2.2} and Lemma \ref{lemma2.5}.
\end{proof}

%%%%%%%%%%%%%%%%%%%%%%%%%%%%%%%%%%%%%%%%%%%%%%%%%%%%%%%%%%%%%%%%%%%%%%%%%%%%%%%%%%%%%%%%%%%%%%%%%%%%%%%%%%%%%%%%%%%%%%%%%%%%%%%%%%%
\section{Extremal problems on coefficients and Clenshaw's conjecture}

First let us recall how to solve the following classical
problem in complex function theory. Let
$\mu_0, ... , \mu_l \in \mathbb C$ be given. Among all
$f \in H^{\infty}$ with
\begin{equation}\label{new2}
   f(z) = \sum_{j=0}^{\infty} a_j z^{j} \ {\rm and \ } |f(z)|
   \leq 1 \ {\rm for \ } |z| < 1
\end{equation}
find
\begin{equation}\label{new47}
\eta_l := \max_ {a_j} | \mu_l a_0 + \mu_{l-1} a_1 + ... + \mu_0
a_l |.
\end{equation}

By \cite{Ger, Lemma } for given $\mu_0,...,\mu_l
\in \mathbb C$ there exist polynomials $s_{l - \nu}(z)$ and $r_\nu(z)$
of degree $l - \nu$ and $\nu$ respectively, which have all zeros in
$|z|<1$ such that
\begin{equation}\label{56n}
\begin{aligned}
  \mu_0 + \mu_1 z + ... + \mu_l z^l + ...
     & = (s^{*}_{l - \nu}(z))^2 r_\nu^*(z) r_\nu(z) \\
     & =: h_{2l}^*(z)
\end{aligned}
\end{equation}
F. Riesz \cite{Rie} has shown that among all functions $g \in H^1$
with $g(z) = \mu_0 + \mu_1 z + ... + \mu_l z^l + ... $ the mean
modulus $\frac{1}{2\pi} \int_0^{2\pi} |g(e^{i\varphi})| d\varphi$
is minimal for $h_{2l}^*(z).$ With the help of this result Szasz
\cite{} derived that
\begin{equation}\label{eta}
  \eta_l = \frac{1}{2\pi} \int_0^{2\pi} |h_{2l}^*(e^{i\varphi})| d\varphi
\end{equation}
and that equality in \eqref{new47} is attained by
\begin{equation}\label{eq-70}
   f(z) = e^{i\gamma} \frac{s_{l - \nu}(z)}{s_{l - \nu}^*(z)}
        = c_0 + c_1 z + ... + c_l z^l + ...
\end{equation}
We mention that \eqref{eta} holds, since Szasz's problem
\eqref{new2}-\eqref{new47} and the problem considered by F. Riesz
are dual problems, see \cite{Gar}.

The simple case $\nu = 0$ appears if the expansion in $z = 0$
\begin{equation*}
   \sqrt{\mu_0 + \mu_1 z + ... + \mu_l z^l} =
   \sum_{j = 0}^\infty \lambda_j z^j
\end{equation*}
is such that
\begin{equation*}
   s_l^* (z) = \sum_{j = 0}^l \lambda_\nu z^\nu
\end{equation*}
has no zero in $|z| < 1.$ Then
\begin{equation*}
    \eta_l = \sum_{j = 0}^l |\lambda_j|^2 = \frac{1}{2 \pi} \int_0^{2 \pi}
    |s_l(e^{i \varphi})|^2 d \varphi
\end{equation*}
and
\begin{equation*}
   f(z) = e^{i \gamma} \frac{s_l(z)}{s_l^*(z)}
\end{equation*}
is the function for which the maximum \eqref{new47} is attained.

The special case $\mu_j = 1, j=0,...,l,$ i.e. the determination of
the maximum of the sum of coefficients (which fits into the case just
considered), has been first considered and solved by an ad hoc method
in the celebrated paper by Landau \cite{}. Landau has shown that
\begin{equation}\label{D3}
  |a_0 + a_1 + ... + a_l| \leq 1 + \sum_{j=1}^{l}
  \left( \frac{1.3. \ ... \ . (2j -1)}{2.4. \ ... \ .2j} \right)^2
\end{equation}
and that equality holds only for
\begin{equation}\label{D4}
  F(z) = e^{i \gamma} \frac{\sum\limits_{\nu=1}^{l} (-1)^\nu
  {\binom {-\frac 12}\nu} z^{l - \nu} }{ \sum\limits_{\nu=1}^{l}
  (-1)^\nu \binom {-\frac 12}\nu z^{\nu} }
\end{equation}
where $\gamma \in \mathbb R.$

The other simple case $\nu = l$ appears if the $\mu_j$'s are such that
\begin{equation}
    {\rm Re \ } \{ \mu_0 e^{i l \varphi} + ... + \mu_{l - 1} e^{i \varphi}
    + \frac{\mu_l}{2} \} \geq 0 {\rm \ on \ } [0, 2\pi]
\end{equation}
Then
\begin{equation*}
    \eta_l = \mu_l
\end{equation*}
and equality is attained in \eqref{new47} for $f(z) = \varepsilon,$
$|\varepsilon| = 1.$

Here we study the following problem: Let $\mu_0,...,\mu_l \in \mathbb C$
be given. How large can be $|\sum_{j = 0}^l \mu_{l - j} \tau_j|$ if
$|| {\rm Re \ } \{ \sum_{j = 0}^n \bar{\tau}_j z^{n - j} \} || \leq 1$
and $n$ is large. In other words: among all upper bounds $L$ such that
for each $n \in \mathbb N$ and for every $(\tau_0,...,\tau_l,..., \tau_n)
\in {\mathbb C}^n$
\begin{equation}\label{eq-m1}
     |\sum_{j = 0}^l \mu_{l - j} \tau_j| \leq L
     || {\rm Re \ } \{ \sum_{j = 0}^n \bar{\tau}_j e^{i(n - j) \varphi} \} ||
\end{equation}
find the least upper bound.

We point out that by \eqref{new2} and \eqref{new47} $\eta_l$ is such an upper bound
$L$ in \eqref{eq-m1}. As we show in the next theorem it is even the least upper bound.

\begin{thm}\label{thm-mi2}
The least upper bound in \eqref{eq-m1} is given by $\eta_l(\mu_0,...,\mu_l).$
\end{thm}

\begin{proof}
First we note that by \eqref{new2} for every $n \in \mathbb N$
\begin{equation*}
    |\sum_{j = 0}^l \mu_{l - j} \tau_j| \leq \eta_l {\rm \ \ \ if \ \ \ }
    || {\rm Re \ } \{ \sum_{j = 0}^n  \bar{\tau}_j e^{i (n - \varphi )}\} || \leq 1
\end{equation*}
Thus we have to show that the upper bound $\eta_l$ cannot be improved.
By \eqref{eq-70}  we know that $\max\limits_{\tau_j} | \sum_{j = 1}^l \mu_{l - j} \tau_j | =
\eta_l$ is attained for a Blaschke product
\begin{equation}
    e^{i \kappa} \frac{s_{l - \nu}(z)}{s_{l - \nu}^*(z)} =
    c_0 + c_1 z + c_2 z^2 + ... + c_l z^l + ...
\end{equation}
Thus it follows by Theorem \ref{three}
\begin{equation*}
    {\mathcal Z}_n(\varphi; \bar{c}_0,...,\bar{c}_l) =
    {\rm Re \ } \{ z^n e^{i \kappa} \frac{s_{l - \nu}(z)}{s_{l - \nu}^*(z)} \}
    + O(\tilde{r}^n)
\end{equation*}
with, setting ${\boldsymbol c} = (c_0,...,c_l),$
\begin{equation*}
    1 - \varepsilon_n \leq
    || {\mathcal Z}_n(\varphi; \bar{\boldsymbol c}) ||
    \leq 1 + \varepsilon_n,
    {\rm \ i.e. \ }
    \frac{1}{1 - \varepsilon_n}
    \leq
    E_n(\bar{\boldsymbol c})
    \leq
    \frac{1}{1 + \varepsilon_n}
\end{equation*}
and $\varepsilon_n \to 0$ geometrically. Hence
\begin{equation*}
    \frac{|\sum_{j = 0}^l \mu_{l - j} c_j|}
         {|| {\mathcal Z}_n(\varphi; \bar{\boldsymbol c}) ||}
    \leq
    \frac{\eta_l}{1 + \varepsilon_n}
    {\rm \ with \ } \varepsilon_n \to 0
    {\rm \ geometrically}
\end{equation*}
which proves the theorem.
\end{proof}

\begin{cor}\label{cor6.2}
Let $\mu_0,...,\mu_l \in \mathbb C$ be given. Then for every
$(\tau_0,...,\tau_l) \in {\mathbb C}^l$ and for every $n \in \mathbb N$
\begin{equation}
    || {\rm Re \ } \{ \sum_{j = 0}^l  \mu_{l - j} \tau_j e^{i(n - j) \varphi} \}||
    \leq
    \eta_l
    || {\rm Re \ } \{ \sum_{j = 0}^n  \bar{\tau}_j e^{ i (n - j) \varphi} \} ||
\end{equation}
and the constant $\eta_l$ cannot be improved.
\end{cor}

\begin{proof}
For any $\psi \in [0,2\pi]$ and for every $n \in \mathbb N$
\begin{equation*}
\begin{split}
      \sum_{j = 0}^l \mu_{l - j} \tau_j e^{i(n - j)\psi}
          & \leq \eta_l \min_{\substack{\tau_j \\ l + 1 \leq j \leq n}}
                 || {\rm Re \ } \{ \sum_{j = 0}^n \bar{\tau}_j e^{- i (n - j)\psi} e^{i(n - j)\varphi} \} || \\
          & = \eta_l \min_{\substack{\tau_j \\ l + 1 \leq j \leq n}}
              || {\rm Re \ } \{ \sum_{j = 0}^n \bar{\tau}_j e^{i(n - j)\varphi} \} || \\
\end{split}
\end{equation*}
where in the first inequality the upper bound $\eta_l$ is best possible
by Theorem \ref{thm-mi2} and the last equality follows by $2 \pi$-periodicity.
\end{proof}

As a consequence of Corollary \ref{cor6.2} and Landau's results \eqref{l1}
we obtain a proof of Clenshaw's conjecture \cite{} from the sixties of the last century.

\begin{notation}
For given ${\boldsymbol \tau} = (\tau_0,..., \tau_l) \in {\mathbb C}^{l + 1}$ let
\begin{equation*}
     E_n({\boldsymbol \tau}) = || {\mathcal Z}_n(\varphi; {\boldsymbol \tau}) ||_{[0,2\pi]}
\end{equation*}
the minimum deviation.
\end{notation}

\begin{thm}
Clenshaw's conjecture holds, that is, for any $\tau \in {\mathbb R}^{l + 1}$
\begin{equation}\label{eq-cle}
     \lim_{n \to \infty}
     \left(
          || \sum_{j = 0}^l \tau_j \cos (n - j) \varphi ||/E_n(\tau)
     \right)
     \leq
          1 + \sum_{j = 1}^l
          \left( \frac{1. 3. \cdots (2 j - 1)}{2. 4. ... 2 j} \right)^2
\end{equation}
where in \eqref{eq-cle} equality is attained.
\end{thm}

\begin{proof}
Put $\mu_l = \mu_{l - 1} = \cdots = \mu_0 = 1.$ Then, as in the proof of Corollary
\ref{cor6.2}, we have for $\psi \in [- \pi, \pi]$ or equivalently for any $-\psi \in [- \pi, \pi]$
\begin{equation*}
\begin{aligned}
     & || {\rm Re \ } \{ \sum_{j = 0}^l \bar{\tau}_j  e^{i (n - j) \psi} \} || \leq
       | \sum_{j = 0}^l \tau_j e^{- i (n - j) \psi } | \leq \\
     & \eta_l((1,\cdots, 1))
       \min_{l+1 \leq j \leq n} || {\rm Re \ } \{ \sum_{j = 0}^n \bar{\tau}_j  e^{i (n - j) \varphi} \} || =
       \eta_l((1,\cdots, 1))E_n(\bar{\boldsymbol \tau})  \\
\end{aligned}
\end{equation*}
\end{proof}

\begin{thm}\label{thm3.4}
For any $\varepsilon > 0$ and all $n, \ n \geq n_0(\varepsilon),$
any $\mu_0,...,\mu_n, \ \tau_0,...,\tau_n \in \mathbb C$
\begin{equation}\label{C.4}
\begin{aligned}
   4 ||{\rm Re} \ \{ \sum\limits_{j=0}^{l} \mu_{l-j} \tau_j
   e^{i(n-j)\varphi}\}|| \leq
   & \left( ||{\rm Re} \ \{ \sum\limits_{j=0}^{n} \mu_{l-j} \tau_j
   e^{i(n-j)\varphi}\}||_1 + \varepsilon \right) \\
   & \ \ \ .||{\rm Re} \ \{ \sum\limits_{j=0}^{n} \bar{\tau}_j
   e^{i(n-j)\varphi}\}||,
\end{aligned}
\end{equation}
where $||f(\varphi)||_1 = \int_{0}^{2\pi} |f(\varphi)|d\varphi.$
\end{thm}

If the $\mu_j$'s and $\tau_j$'s are real for $j=0,...,l$ then the
estimate \eqref{C.4} cannot be improved in the sense, that for
given $\mu (\taub)$ there exists a $\taub (\mu)$ such that
equality is attained $n \to \infty.$

\begin{proof}{Proof of Theorem \ref{thm3.4}:} \\
Case 1. $s_l(z)$ satisfies condition ??.\\
By ??
\begin{equation}
    z^{n-2l} s_l^2(z) = \sum_{j=0}^{l} \mu_j z^{n-j} + O(z^{n-l-1}).
\end{equation}
We claim that
\begin{equation}\label{newtilda1}
  \begin{split}
    & \min\limits_{t \in {\frak T}_{n-l-1}} \int_0^{2\pi}
    |{\rm Re} \{ \sum\limits_{j=0}^l \mu_j z^{n-j} \} + t(\varphi)|d\varphi
    = \int_0^{2\pi} |{\rm Re} \{ z^{n-2l} s_l^2(z) \} | d\varphi
    \\
    & = \frac{2}{\pi} \int_0^{2\pi} |s_l(z)|^2 d\varphi = 4
    \sum_{j=0}^l |\lambda_j|^2
  \end{split}
\end{equation}
from which the assertion follows by recalling \eqref{C.1}. Re $\{
z^{n-2l} s_l^2(z) \}$ deviates least from zero with respect to
$L_1$-norm on $[0,2\pi]$ among all trigonometric polynomials of
the form Re $\{ \sum\limits_{j=0}^{l} \mu_j z^{n-j} +
t(\varphi)\},$ $t \in {\frak T}_{n-l-1}.$ As it is well known it
suffices to show that
\begin{equation}\label{newtilda2}
  \int_0^{2\pi} e^{ik\varphi} \ {\rm sgn \ Re }
  \ \{ e^{i(n-2l)\varphi} s_l^2(e^{i \varphi})
  \} d\varphi = 0 \ {\rm for} \ k = 0, ... , n-l-1.
\end{equation}
Obviously
$$
\begin{aligned}
   & {\rm sgn \ Re} \ \{ e^{i(n-2l)\varphi} s_l^2(e^{i\varphi}) \} =
     {\rm sgn \ Re} \ \{ e^{i(n-l)\varphi}
     \frac{s_l(e^{i\varphi})}{s_l^*(e^{i\varphi})} \}\\
   & = {\rm sgn} \ \cos ((n-l)\varphi + \Phi(\varphi)) =
     \frac{4}{\pi}{\rm Re}\
     \arctan e^{i((n-l)\varphi + \Phi(\varphi))} \\
   & = \frac{4}{\pi} {\rm Re} \ \arctan z^{n-l} \frac{s_l(z)}{s_l^*(z)}
     = \frac{4}{\pi} z^{n-l}
     \frac{s_l(z)}{s_l^*(z)}+ O(z^n),
\end{aligned}
$$
where we used the fact that, $z=e^{i\varphi},$
$$ \frac{4}{\pi} {\rm Re} \ \arctan z = \frac{4}{\pi} \arg
\frac{i-z}{i+z} = {\rm sgn } \ \cos \varphi;$$
hence \eqref{newtilda2} follows and thus the first equality in
\eqref{newtilda1} is proved. Next we observe that
$$
\begin{aligned}
    & | {\rm Re} \{ z^{n-2l} s_l^2(z) \} | = | s_l(z) |^2 |
      {\rm Re} \{ z^{n-l} \frac{s_l(z)}{s_l^*(z)} \} | \\
    & = |s_l(z)|^2 |\cos ((n-l)\varphi + \Phi(\varphi))| =
      |s_l(z)|^2 {\rm Re} \{ \frac {2}{\pi} + O(z^{2(n-l)}) \}
\end{aligned}
$$
where the last equality follows by Fourier-expansion.
\end{proof}

%%%%%%%%%%%%%%%%%%%%%%%%%%%%%%%%%%%%%%%%%%%%%%%%%%%%%%%%%%%%%%%%%%%%%%%%%%%%%%%%%%%%%%%%%%%%%%%%%%%%%%%%%%

\end{document}